\documentclass[11pt]{amsart} 

\usepackage[utf8]{inputenc}	
\usepackage[T1]{fontenc}	
\usepackage{lmodern}		
\usepackage{graphicx,amsmath,amsfonts,amssymb,setspace,subfigure,color}
\newtheorem{theorem}{Theorem}[section]
\newtheorem{corollary}[theorem]{Corollary}
\newtheorem{definition}[theorem]{Definition}

\newtheorem{lemma}[theorem]{Lemma}
\newtheorem{proposition}[theorem]{Proposition}
\newtheorem{remark}[theorem]{Remark}

\usepackage{accents}
\theoremstyle{plain}

\numberwithin{equation}{section}

\allowdisplaybreaks		
\begin{document}

\title[Pleijel-type theorem]{A Pleijel-type theorem for the quantum harmonic oscillator}
\author{Philippe Charron}

\date{\today}

\begin{abstract}
	We prove a Pleijel-type theorem for the asymptotic behaviour of the number of nodal domains of eigenfunctions of the quantum harmonic oscillator in any dimension.
\end{abstract}

\maketitle

\section{Introduction and main results}

\subsection{Pleijel's nodal domain theorem}

Let $\Omega \subset \mathbb{R}^n$ be a bounded domain. Let $\lambda_1 < \lambda_2 \leq \lambda_3 \ldots$ be the eigenvalues of the Dirichlet Laplacian in $\Omega$ and let ${\left\lbrace f_i \right\rbrace}_{i \geq 1}$ be an orthogonal basis of eigenfunctions associated with those eigenvalues.

\bigskip

Recall that a nodal domain of a function is a connected component of the complement of the zero-set of that function. Let $\mu(f)$ be the number of nodal domains of the function $f$.

\bigskip

Recall that Courant's nodal domain theorem states that $\mu(f_k)\leq k$. In 1956, Pleijel found a better estimate when eigenvalues tend to infinity. There exists a constant $\gamma(n) < 1$ that depends only on the dimension such that:

\begin{equation}
\label{pleijel1}
\limsup\limits_{k \to \infty} \frac{\mu(f_k)}{k} \leq \gamma(n) = \frac{2^{n-2}n^2 \Gamma(n/2)^2}{(j_{\frac{n}{2}-1})^n} \, .
\end{equation}

Here, $j_{\frac{n}{2}-1}$ is the first zero of the Bessel function of the first kind $J_{\frac{n}{2}-1}$.

\bigskip

This constant is strictly decreasing with $n$ (see \cite[p. 10]{HelfferPersson}). Here are the first few values: $\gamma(2)=0.69166, \gamma(3)= 0.455945, \gamma(4)= 0.296901, \gamma(5)=0.19294$.

\bigskip

\begin{remark}
	\textnormal{
  This result has been proved in the case of the Neumann Laplacian in dimension $2$ for piecewise analytic domains in \cite{Polterovich}. It is still unknown if the result holds in the Neumann case in higher dimensions. Recent efforts (\cite{Bourgain}, \cite{Steinerberger}) have been made to improve the estimate in dimension~$2$.}
\end{remark}

\bigskip

\subsection{Quantum harmonic oscillator}

Our goal is to study the nodal domains of eigenfunctions of the quantum harmonic oscillator. 

\bigskip

The quantum harmonic oscillator is first defined on $S(\mathbb{R}^n)$ by:

\begin{align}
\label{Hf1}
H: S(\mathbb{R}^n) \to S(\mathbb{R}^n) \, , \nonumber \\
H f = -\Delta f + V(x) f \, .
\end{align}

Here, $V$ is a positive-definite quadratic form and $S(\mathbb{R}^n)$ denotes the Schwartz space of rapidly decaying functions over $\mathbb{R}^n$.

\bigskip

There exists a unique self-adjoint extension of $H$ over $L^2(\mathbb{R}^n)$, which will be denoted by $\mathbf{H}$. However, there exists a basis of $L^2(\mathbb{R}^n)$ consisting of eigenfunctions of $\mathbf{H}$ which are all in $S(\mathbb{R}^n)$. 

\bigskip

The quantum harmonic oscillator can be viewed as a Schrödinger operator with potential $V(x)$. It has two properties that make it particularly interesting. Its spectrum is discrete since $\lim\limits_{|x| \to \infty}V(x)= +\infty$ (see \cite{Simon}) and its eigenfunctions can be computed explicitly.

\bigskip

There exists an orthogonal basis  $y_1, y_2, \ldots, y_n$ of $\mathbb{R}^n$ and constants \linebreak $a_1, a_2, \ldots, a_n > 0$ such that $V(x)= \sum\limits_{i=1}^n a_i^2 y_i^2$. The Laplacian is invariant under orthogonal changes of the basis. Therefore, if we wish to study the nodal domains of the eigenfunctions of the harmonic oscillator, we can restrict ourselves to potentials of the following form: 

\begin{equation}
\label{H}
V(x) = \sum\limits_{i=1}^n a_i^2 x_i^2 \, .
\end{equation}

If all the coefficients $a_i$ are equal, the quantum harmonic oscillator $H$ is called \textit{isotropic}.
\bigskip

A basis in $L^2(\mathbb{R}^n)$ of the eigenfunctions of $H$ is given by

\begin{equation}
\label{fonctionpropre}
f_{k_1, \ldots, k_n}(x) = \prod\limits_{i=1}^n e^{\frac{-a_i x_i^2}{2}} H_{k_i}(\sqrt{a_i}x_i) \, .
\end{equation}

Here, $H_n$ denotes the $n$-th Hermite polynomial, see \cite{Szego}.

\bigskip

The corresponding eigenvalues are given by $\lambda_{k_1 \ldots k_n}= \sum\limits_{i=1}^n a_i (2k_i + 1)$.

\bigskip
Note that Courant's theorem holds for $H$ by a straightforward adaptation of the argument for the Laplacian. Two slightly improved results in the isotropic case can be found in \cite{HelfferBerard} and \cite{Leydold}.

\subsection{Main result}

The main result of this paper is

\begin{theorem}
\label{theoremprincipal}

Let $H$ be the quantum harmonic oscillator (\ref{Hf1}).

\bigskip

The number $\mu(f_k)$ of nodal domains of the $k$-th eigenfunction of $H$ satisfies:

\begin{equation}
\label{ineqprincipale}
\limsup\limits_{k\rightarrow \infty} \frac{\mu(f_k)}{k} \leq \gamma(n) \, .
\end{equation}
\end{theorem}

The constant $\gamma(n)$ is the same as in equation (\ref{pleijel1}).

\subsection{Eigenvalue multiplicities}
If the coefficients in (\ref{H}) are rationally independent, the eigenvalues of $H$ are simple. Recall that $a_1, a_2, \ldots, a_n$ are rationally dependent if the only integers $k_1, k_2, \ldots, k_n$ that satisfy $a_1 k_1 + a_2 k_2 + \ldots + a_n k_n = 0$ are identically zero. In this case, we can compute the number of nodal domains of each eigenfunction since it is always a product of polynomials in one variable and obtain:

\begin{theorem}
	\label{theoremeprincipal2}
	Let $H$ be the quantum harmonic oscillator (\ref{Hf1}) with the coefficients $a_1, a_2, \ldots, a_n$ rationally independent.
	
	\bigskip
	
	The number $\mu(f_k)$ of nodal domains of the $k$-th eigenfunction of $H$ satisfies:

	\begin{equation}
	\label{ineqprincipale2}
	\limsup\limits_{k\rightarrow \infty} \frac{\mu(f_k)}{k} = \frac{n!}{n^n} \, .
	\end{equation}
\end{theorem}

 However, if some coefficients are rationally dependent, the eigenspace associated with an eigenvalue may have dimension greater than one and we need to deal with linear combinations of eigenfunctions.

\bigskip

For instance, in the isotropic case in $\mathbb{R}^n$, which is the most widely studied, the eigenvalues are $\lambda_k = 2j + n$ for all $k \in \left[ \binom{n+j-1}{j-1} +1, \binom{n+j}{j} \right]$. 

\bigskip

Hence, the multiplicities grow to infinity. It is therefore hard to compute the number of nodal domains of the eigenfunctions directly. In this paper, we present a different approach that covers all cases.

\subsection{Sketch of the proof of Theorem \ref{theoremprincipal}}

When we analyse Pleijel's original proof of the theorem in the case of the Laplacian with Dirichlet boundary conditions on an Euclidian domain $\Omega$, the main idea is to give a lower bound on the area of each nodal domain using Faber-Krahn's inequality. We then divide the area of $\Omega$ by this lower bound and apply Weyl's law to get the final inequality.

\bigskip

If we try to use the same argument for the quantum harmonic oscillator, there is an obstacle: we are considering functions over $\mathbb{R}^n$, which has infinite volume. We must therefore find a way to resolve this issue.

\bigskip

We first show that any nodal domain must intersect the classically allowed region $\left\lbrace  V(x) < \lambda \right\rbrace $ (see \cite{Hall}), which in our case is the interior of an ellipsoid.

\bigskip

We then divide this ellipsoid into regions called generalized annuli (see Definition \ref{generalizedannulus}). This is the main new idea, which lets us bound the number of nodal domains. We use a theorem of Milnor on the Betti numbers of sublevel sets of real polynomials in order to give an upper bound on the number of nodal domains that intersect more than one generalized annulus. Finally, we use Faber-Krahn's inequality to get lower bound on the area of each nodal domain located in each generalized annulus.

\section{Proof of theorem \ref{theoremprincipal}}

\subsection{Eigenvalues and eigenfunctions of $H$}

Recall that every eigenfunction of $H$ is of the form $f= \prod\limits_{i=1}^n e^{\frac{-a_i x_i^2}{2}}g(x)$, where $g$ is a polynomial. By slight abuse of notation, we define the degree of an eigenfunction $f$ as the degree of its associated polynomial $g$.

\bigskip

Note that $f_{k_1 \ldots k_n}$ is an eigenfunction of degree $k_1 + \ldots + k_n$ from equation~(\ref{fonctionpropre}).

\begin{remark}
	\textnormal{In the isotropic case, the eigenfunctions are ordered with their degrees as well as their eigenvalues. In the anisotropic case, the degrees of the eigenfunctions may not be strictly increasing.}
\end{remark}

We give upper bounds on the degree of $f_k$:

\begin{align}
\deg(f_k) \leq \max\limits_{\substack{k_1, k_2 \ldots, k_n \in \mathbb{Z}^+\\\sum\limits_{i=1}^n a_i (2 k_i + 1) \leq \lambda_k}} \,  \sum\limits_{i=1}^n k_i \, .
\end{align}

Take $i$ such that $a_i = \min \left\{ a_j, j=1, \ldots, n \right\}$. The maximum is obtained in the previous sum by putting $k_j=0$ when $j \neq i$ and maximizing $k_i$, namely

\begin{align}
\label{degre}
\deg(f_k(x)) \leq \frac{\lambda_k-\frac{n}{2}}{2 \min\limits_{i=1,\ldots, n} a_i} \, .
\end{align}

Let $N(\lambda)$ be the number of eigenvalues of $H$ that are not greater than $\lambda$. We have

\begin{equation*}
N(\lambda) = \mathrm{Card}\left( k_1, k_2 \ldots, k_n \in \mathbb{Z}^+ \, \left| \, \sum\limits_{i=1}^n a_i (2 k_i + 1) \leq \lambda \right. \right) \, .
\end{equation*}

Using the formula for the volume of an $n$-simplex, we obtain the following asymptotics when $\lambda \to +\infty$:

\begin{equation}
\label{anisotropic1}
N(\lambda)= \lambda^n \left( \frac{1}{2^n n! \prod\limits_{i=1}^n a_i} +o_\lambda(1) \right) \, .
\end{equation}

Also, if we put $\lambda = \lambda_k$ in (\ref{anisotropic1}), we get the following:

\begin{align*}
N(\lambda_k)= \lambda_k^n \left( \frac{1}{2^n n!\prod\limits_{i=1}^n a_i} + o(1)\right) \, .
\end{align*}

We remark that $N(\lambda_k) \geq k$ since $\lambda_k$ could have multiplicity greater than one. We can deduce the following:

\begin{align}
\label{lambdan}
\lambda_k^n \leq k\left( 2^n n! \prod\limits_{i=1}^{n} a_i + o(1)  \right) \, .
\end{align}

We can rewrite the previous equation the following way:

\begin{align}
\label{lambdan2}
\lambda_k \leq k^{1/n} \left( \left(2^n n! \prod\limits_{i=1}^{n} a_i\right)^{1/n} + o(1)  \right) \, .
\end{align}

Hence, from (\ref{degre}) and (\ref{lambdan2}) we have the following inequality for the degree of $f_k$:

\begin{align}
\label{degre2}
\deg(f_k(x)) \leq k^{1/n} \left( \frac{ \left(2^n n! \prod\limits_{i=1}^{n} a_i\right)^{1/n}}{2 \min\limits_{i=1\ldots n} a_i} + o_k(1)\right) \, .
\end{align}

\subsection{Unbounded nodal domains}

Let $\Omega$ be an unbounded nodal domain of $f_k$. Since, for all $k$, $f_k \in S(\mathbb{R}^n)$, we have the following equality:

\begin{align}
\label{eq1}
\lambda_k = \frac{\int_\Omega {{|\nabla f_k|}^2} + \int_\Omega {V(x)f_k^2}}{\int_\Omega {f_k^2}} \, .
\end{align}

\begin{lemma}
\label{lemma1}
For each nodal domain $\Omega$, there exists $x \in \Omega$ such that $V(x) \leq \lambda_k$
\end{lemma}

\begin{proof}

If for all $x \in \Omega$, $V(x) > \lambda_k$, then

\begin{align}
\lambda_k \, = \frac{\int_\Omega {{|\nabla f_k|}^2} + \int_\Omega {V(x)f_k^2}}{\int_\Omega {f_k^2}} \geq \frac{\int_\Omega {V(x)f_k^2}}{\int_\Omega {f_k^2}} > \frac{\int_\Omega {\lambda_k f_k^2}}{\int_\Omega {f_k^2}} = \lambda_k \, ,
\end{align}

hence a contradiction.

\end{proof}

\bigskip

Therefore, every unbounded nodal domain intersects the following ellipsoid: $$\left\{ x \in \mathbb{R}^n \, | \, V(x) = \lambda_k \right\} \, .$$

\bigskip

\subsection{Bounded nodal domains}

Let us now study the bounded nodal domains. Since $$f_k(x) = e^{-\frac{\sum\limits_{i=1}^n a_i x_i^2}{2}}g_k(x) \, ,$$ with $g_k(x)$ a polynomial, the nodal domains of $f_k$ are the same as the nodal domains of $g_k$. First, let us define a specific subset of $\mathbb{R}^n$. 

\begin{definition}
	\label{generalizedannulus}
Let $0 \leq b < B < +\infty$. We define a \textit{generalized annulus} as

\begin{equation}
\left\{ (x_1, \ldots, x_n) \in \mathbb{R}^n \, | \,  b < \sum\limits_{i=1}^n a_i x_i^2 < B  \right\} \, .
\end{equation}

\end{definition}

We have just shown that every nodal domain intersect the interior of the ellipsoid described above. We divide this region in a given number of generalized annuli. The number of generalized annuli will depend on the eigenfunction. The number of generalized annuli is quite important since we count the number of nodal domains in two ways: those that are contained in one generalized annulus and those that intersect more than one generalized annulus. Having more generalized annuli will restrict the former and increase the latter, and conversely. 

\bigskip

Let $M = M(\lambda_k)$ be the number of generalized annuli for a given eigenfunction. We will give an explicit formula for $M$ later.

\bigskip

Now, let us define the following sets:

	\begin{definition}
		\label{ai}
		$$A_i= \left\{ \Omega \, \left| \, \, \forall x \in \Omega, (\frac{(i-1)}{M})^{2/n} {\lambda_k} \leq {V(x)} < (\frac{i}{M})^{2/n} {\lambda_k} \right. \right\} \, .$$ Here, $i$ can take the values $1,2, \ldots, M$.
	\end{definition}

	\begin{definition}
		\label{bj}
		$$B_j = \left\{ \Omega \, \left| \, \, \Omega \cap \left\lbrace V(x) = (\frac{j}{M})^{2/n} {\lambda_k}\right\rbrace \neq \emptyset   \right. \right\} \, .$$ Again, $j$ can take the values $1,2, \ldots, M$.
	\end{definition}

In fact, every nodal domain, bounded or unbounded, is included in one of those sets. Indeed, as shown in Lemma \ref{lemma1}, for each nodal domain $\Omega$, there exists $x \in \Omega$ such that $ V(x) \leq \lambda_k$. Hence, by the connectedness of each nodal domain, it belongs to one of the $A_i$ or $B_j$.

\subsection{Nodal domains intersecting more than one generalized annulus}

Let $f: \mathbb{R}^n \to \mathbb{R}$ be a polynomial of degree $k$ in $n$ variables. We wish to give an upper bound on the number of nodal domains of $f$ on the unit $n$-ball. Let $G(n,d)= (2+d)(1+d)^{n-1}$.

\bigskip

Let $F^{+} = \left\{ x \in B^n \, | \, f(x) > 0 \right\}$. First, we show that the number of connected components of $F^+$ has an upper bound that depends only on the degree of $f$. We can find the following result in \cite{Milnor}:

\begin{theorem}[Milnor]
\label{theorembettipolynome}
Let $f$ be a real polynomial of degree $d$ in $n$ variables. We define $P$ as follows:

\begin{equation*}
P = \left\{ x \in B^n \, | \, f(x) \geq 0 \right\} \, .
\end{equation*}

Then the first Betti number of $P$ is not greater than $G(n,d)$.
\end{theorem}

Recall that the first Betti number of a manifold is equal to the number of its connected components.

\begin{remark}
	\textnormal{
		The original theorem gives an upper bound on the sum of the Betti numbers of a real algebraic manifold. Moreover, the Betti numbers are all nonnegative. Also, we could not find a similar result for the sum of the Betti numbers of $\left\{ x \in B^n \, | \, f(x) > 0 \right\}$. Hence, we must add a few more arguments to complete the proof.}
\end{remark}

Let $P_m = \left\{ x \in B^n \, | \, f(x) \geq 1/m \right\}$.

\bigskip

Then, the number of connected components of $P_m$ is not greater than $G(n,d)$. Furthermore, $F^{+} = \lim\limits_{m\rightarrow \infty} P_m$.

\begin{lemma}
\label{F+}
The number of connected components of $F^{+}$ is not greater than $G(n,d)$.
\end{lemma}

\begin{proof}
Suppose that $F^{+}$ has more than $G(n,d)$ connected components. \linebreak Choose connected components $\left\{a_i\right\}, i=1,2, \ldots, G(n,d) +1$ of $F^+$. Take $s_i \in a_i$ such that for all $x \in a_i$, $f(x) \leq f(s_i)$. We can always find such $s_i$ by the compactness of $\overline{a_i}$ and the continuity of $f$.

\bigskip

Now, define $S = \min \left\{f(s_i), i=1,2, \ldots, G(n,d)+1 \right\}$. There exists $m \in \mathbb{N}$ such that $1/m < S$. For each connected component $a_i$, there exists a connected component $b_i \subset P_m$ such that $b_i \subset a_i$. However, that would imply that $P_m$ has at least $G(n,d)+1$ connected components, which would contradict Theorem \ref{theorembettipolynome}.

\end{proof}

We can now give an upper bound on the number of nodal domains of a polynomial on $B^n$.

\begin{proposition}
	\label{theoremboule}
	Let $f: \mathbb{R}^n \to \mathbb{R}$ a polynomial of degree $d$. The number of nodal domains of $f$ in $B^n$ is not greater than $2G(n,d)$.
\end{proposition}

\begin{proof}

Let $F^{-}= \left\{ x \in B^n \, | \, f(x)<0 \right\}$. Clearly, $F^{+}$ and $F^{-}$ are disjoint. By the same argument as before, the number of connected components of $F^{+} \bigcup F^{-}$ is not greater than $2 G(n,d)$. 

\end{proof}

\bigskip

Now, let us find an upper bound on the number of nodal domains of the restriction of a polynomial in $n$ variables to $S^{n-1}$.

\begin{proposition}
\label{theoremsphere2}
Let $f: \mathbb{R}^n \to \mathbb{R}$ be a polynomial of degree $d$. Then, the number of nodal domains of the restriction of $f$ to $\mathbb{S}^{n-1}$ is not greater than $2^{2n-1} d^{n-1}$.
\end{proposition}

\begin{proof}

On $\mathbb{S}^{n-1}$, we can use the relation $x_1^2 = 1 - \sum\limits_{i=2}^{n} x_i^2$. We can then rewrite $f$ in the following form: 

\begin{equation*}
f(x_1,x_2,\ldots,x_n) = g(x_2,x_3,\ldots,x_n) + x_1 \cdot h(x_2,x_3,\ldots,x_n) \, .
\end{equation*}

\bigskip

Here, $g$ is a polynomial of degree at most $d$ and $h$ is a polynomial of degree at most $d-1$.

\bigskip

Now, define ${\bar{f}}: \mathbb{R}^n \to \mathbb{R}$ as follows: $$\bar{f}(x_1,x_2,\ldots,x_n)=g(x_2,x_3,\ldots,x_n) - x_1 \cdot h(x_2,x_3,\ldots,x_n)$$ 

On $\mathbb{S}^{n-1}$, we have the following:

\begin{align}
f \cdot \bar{f} = g^2(x_2,x_3,\ldots,x_n) + (\sum\limits_{i=2}^{n}x_i^2 -1) \cdot h^2(x_2,x_3,\ldots,x_n) \, .
\end{align}

Hence, $f\bar{f}$ is a polynomial of degree $2d$ in only $n-1$ variables.

\bigskip

Define $\phi$ by:

\begin{align*}
\phi: B^{n-1} \to \left\{x \in \mathbb{S}^{n-1} \, | \, x_1 > 0 \right\} \, , \\
\phi(x_2, \ldots, x_n) = (\sqrt{1- \sum\limits_{i=2}^{n}x_i^2}, x_2,\ldots, x_n) \, .
\end{align*}

\bigskip

Let $\mathbf{f} : B^{n-1} \to \mathbb{R}$, $\tilde{f} = (f\bar{f}) \circ \phi$. It is the restriction of a polynomial of degree $2d$ in $n-1$ variables on the unit ball in $\mathbb{R}^{n-1}$. By Proposition \ref{theoremboule}, the number of nodal domains of $\mathbf{f}$ in $B^{n-1}$ is not greater than $(2+2d)(1+2d)^{n-2}$.

\bigskip

We have the following for $d \geq 1$:

\begin{align}
(2+2d)(1+2d)^{n-2} < 2^{2n-2}d^{n-1} \, .
\end{align}

\bigskip

The function $\phi$ projects the nodal domains of $\mathbf{f}$ onto $\mathbb{S}^{n-1}$. Hence, the number of nodal domains of $f\bar{f}$ in $\left\{x \in \mathbb{S}^{n-1} \, | \, x_1 > 0 \right\}$ is not greater than $2^{2n-2}d^{n-1}$.

\bigskip

By the same argument, the number of nodal domains of $f\bar{f}$ \linebreak in $\left\lbrace x \in \mathbb{S}^{n-1} \, \left| \, x_1 < 0 \right. \right\rbrace $ is not greater than $2^{2n-2}d^{n-1}$. Furthermore, each nodal domain is either located in the upper part of the $n$-sphere, the lower part of the $n$-sphere or both. Since the number of nodal domains of $f$ is not greater than the number of nodal domains of $f\bar{f}$, we conclude the proof.

\end{proof}

By rescaling variables, we can easily prove the following corollary:

\begin{corollary}
	\label{corollary1}
	Let $a \in \mathbb{R}, a > 0$ and let $f: \mathbb{R}^n \to \mathbb{R}$ be a polynomial of degree $d$. Then, the number of nodal domains of the restriction of $f$ on $\left\lbrace V(x)=a\right\rbrace $  is not greater than $2^{2n-1} d^{n-1}$.
\end{corollary}

\bigskip

We can now give an upper bound on the number of nodal domains that intersect more than one generalized annulus.

\begin{lemma}
\label{lemmabl}
There exists $C > 0$ such that for all $k$,
$$\mathrm{Card} \left(\bigcup\limits_{j=1}^{M}B_j \right) \leq C M k^{\frac{n-1}{n}} \, .$$
\end{lemma}

\begin{proof}

Recall definition \ref{bj} for the sets $B_j$. By the Corollary \ref{corollary1}, $\mathrm{Card}(B_j) \leq 2^{2n-1} {\deg(f_k)}^{n-1}$ for $1 \leq l \leq M$. We now have the following inequality:

\begin{align}
\mathrm{Card} \left(\bigcup\limits_{j=1}^{M}B_j \right) \leq M  2^{2n-1} (\deg(f_k))^{n-1} \, .
\end{align}

We replace $\deg(f_k)$ as in equation (\ref{degre2}):

\begin{align}
\mathrm{Card} \left(\bigcup\limits_{j=1}^{M}B_j \right) \leq M 2^{2n-1}  [k^{1/n} \left( \frac{ (2^n n! \prod\limits_{i=1}^{n} a_i)^{1/n}}{2 \min\limits_{i=1\ldots n} a_i} + o_k(1)\right)]^{n-1} \, .
\end{align}

Here, the error term depends only on $k$ (and not $f_k$) so there exists a constant $C > 0$ such that

\begin{align}
\label{ineqbl}
\mathrm{Card} \left(\bigcup\limits_{j=1}^{M}B_j \right)\leq  C M  k^{\frac{n-1}{n}} \, .
\end{align}
\end{proof}

As a result of this, if we take $M$ to grow slower than $k^{\frac{1}{n}}$, the last term will be negligible in our final estimate.

\subsection{Nodal domains contained in a single generalized annulus}

We now turn to the study of nodal domains strictly contained in a single generalized annulus. We first recall Faber-Krahn's inequality in dimension $n$. Let $\Omega$ be a bounded domain of $\mathbb{R}^n$. The first Dirichlet eigenvalue $\lambda_1(\Omega)$ satisfies the following:

	\begin{equation}
	\label{faberkrahn}
	\lambda_1(\Omega) \geq \left( \frac{1}{|\Omega|} \right)^{\frac{2}{n}} {\sigma_n}^{\frac{2}{n}} (j_{\frac{n}{2}-1})^2 \, .
	\end{equation}
	
	As before, $j_{\frac{n}{2}-1}$ is the first zero of the Bessel function of the first kind $J_{\frac{n}{2}-1}$ and $\sigma_n$ is the volume of the unit ball in $\mathbb{R}^n$.

\bigskip

Now, let $\Omega$ be a bounded nodal domain of $f_k$. We have the following inequality:

\begin{equation}
\label{ineqaire3}
\frac{\int_\Omega {{|\nabla f_k|}^2}}{\int_\Omega {f_k^2}} \geq \left(\frac{1}{|\Omega|}\right)^{\frac{2}{n}}{\sigma_n}^{\frac{2}{n}}{(j_{\frac{n}{2}-1})^2} \, .
\end{equation}

Recall definition \ref{ai} for the sets $A_i$, as well as equation (\ref{eq1}). For each $\Omega \in A_i$,

\begin{align}
\label{ineqaire4}
\frac{\int_\Omega {{|\nabla f_k|}^2}}{\int_\Omega {f_k^2}} < \lambda_k - (\frac{i}{M})^{\frac{2}{n}} \lambda_k \, .
\end{align}

Combining (\ref{ineqaire3}) and (\ref{ineqaire4}), we get:

\begin{equation}
\label{ineqvolume}
|\Omega| \geq \frac{\sigma_n (j_{\frac{n}{2}-1})^{n}}{(\lambda_k - (\frac{i}{M})^{\frac{2}{n}}\lambda_k)^{\frac{n}{2}}} \, .
\end{equation}

Let $w_n(x)$ denote the volume of an $n$-ball of radius $x$. The volume of the generalized annulus in which each element of $A_i$ can be found is

\begin{align}
\label{volume10}
\frac{1}{\prod\limits_{i=1}^n a_i} \left( w_n\left(\left(\frac{i}{M}\right)^{\frac{1}{n}} \sqrt{\lambda_k}\right)-w_n\left(\left(\frac{i-1}{M}\right)^{\frac{1}{n}} \sqrt{\lambda_k}\right) \right) \nonumber\\
= \frac{1}{M \prod\limits_{i=1}^n a_i} \sigma_n \lambda_k^{\frac{n}{2}} \, .
\end{align}

Combining (\ref{ineqvolume}) and (\ref{volume10}), we get the following:

\begin{align}
\label{volume2}
\mathrm{Card}(A_i) \leq \frac{\lambda_{k}^n }{ (j_{\frac{n}{2}-1})^n \prod\limits_{i=1}^n a_i} \frac{(1- (\frac{i}{M})^{\frac{2}{n}})^{\frac{n}{2}}}{M} \, .
\end{align}

Using the last inequality, we get the following inequality for the number of elements in every $A_i$:

\begin{align}
\label{sumtointegral}
\mathrm{Card}( \bigcup\limits_{i=1}^{M} A_i) \leq \sum\limits_{i=1}^{M} \frac{\lambda_{k}^n }{(j_{\frac{n}{2}-1})^n \prod\limits_{i=1}^n a_i} \frac{(1- (\frac{i}{M})^{\frac{2}{n}})^{\frac{n}{2}}}{M} \, .
\end{align}

Here, the function $f(x) = (1- x^{\frac{2}{n}})^{\frac{n}{2}}$ is integrable over $[0,1]$, hence the Riemann sum with the partition $\left\{i/{M}\right\}, i=0 \ldots {M}$ converges to the value of the integral when $M$ goes to infinity. 

\bigskip

Choose $M$ such that $M$ goes to infinity with $k$ slower than $k^{\frac{1}{n}}$. Then,

\begin{equation}
\mathrm{Card}( \bigcup\limits_{i=1}^{M} A_i) \leq \frac{\lambda_k^n}{(j_{\frac{n}{2}-1})^n \prod\limits_{i=1}^n a_i} \left( \int_{0}^{1}{(1- x^{\frac{2}{n}})^{\frac{n}{2}} dx} + o_{k}(1) \right) \, .
\end{equation}
\bigskip

We can now compute the integral. Using the substitution $u= x^{\frac{2}{n}}$ (see for example \cite{Gradshteyn}) gives us the following:

\begin{align}
 \int_{0}^{1}{(1- x^{\frac{2}{n}})^{\frac{n}{2}} dx} = \frac{n}{2} \frac{\Gamma(\frac{n}{2}) \Gamma(\frac{n}{2} +1)}{\Gamma(n+1)}= \frac{n^2 \Gamma(n/2)^2}{2^2 n!} \, .
\end{align}

Using equation (\ref{lambdan2}), we get:

\begin{align}
\mathrm{Card}( \bigcup\limits_{i=1}^{M} A_i) \leq  \frac{k\left(2^n n! \prod\limits_{i=1}^{n} a_i + o_k(1)  \right)}{(j_{\frac{n}{2}-1})^n \prod\limits_{i=1}^n a_i} \left( \frac{n^2 \Gamma(n/2)^2}{2^2 n!} + o_{k}(1) \right) \, ,
\end{align}

and finally:

\begin{align}
\label{eqq1}
\mathrm{Card}( \bigcup\limits_{i=1}^{M} A_i) \leq  k  \left( \frac{2^{n-2}n^2  \Gamma(\frac{n}{2})^2}{(j_{\frac{n}{2}-1})^n } + o_k(1)\right)
\end{align}

Combining equation  (\ref{eqq1}) and Lemma \ref{lemmabl} and recalling the fact that we chose $M$ to grow slower than $k^{\frac{1}{n}}$, we get the final inequality:

\begin{align}
\limsup\limits_{k\rightarrow \infty} \frac{N(f_k)}{k} \leq \frac{2^{n-2} n^2 \Gamma(\frac{n}{2})^2}{(j_{\frac{n}{2}-1})^n } \, ,
\end{align}

which completes the proof of Theorem $\ref{theoremprincipal}$.

\section{Proof of Theorem \ref{theoremeprincipal2}}

	Take the coefficients $a_i$ to be rationally independent. Under this assumption, the eigenvalues of $H$ are simple. We know that the $n$-th hermite polynomial has exactly $n$ zeros. Hence, the eigenfunction $f_{k_1, \ldots, k_n}(x)$ has exactly $\prod\limits_{i=1}^n (k_i+1)$ nodal domains. We have the following expression for the maximal number of nodal domains of $f_\lambda$:

	\begin{align}
	\mu(f_\lambda) = \sup\limits_{\substack{k_1, \ldots, k_n\in {\mathbb{Z}^+}\\  \sum\limits_{i=1}^n a_i(2k_i + 1) \leq \lambda}} \prod\limits_{i=1}^n (k_i+1) \, .
	\end{align}
	
	We can give an upper bound on $\mu(f_\lambda)$ in the following way:
	
	\begin{align}
	\label{nodalsup1}
	\mu(f_\lambda) \leq \sup\limits_{\substack{k_1, \ldots, k_n\in {\mathbb{R}^+}\\  \sum\limits_{i=1}^n a_i (2k_i + 1) \leq \lambda}} \prod\limits_{i=1}^n (k_i+1) \, .
	\end{align}
	
	We start by proving the following lemma.
	
	\begin{lemma}
		\label{lemmasup}
		Let $\lambda > 0, a_1, a_2, \ldots, a_n \in \mathbb{R}^+$. We have the following:
		
		\begin{equation}
		\sup\limits_{\substack{k_1, \ldots, k_n\in {\mathbb{R}^+}\\ \sum\limits_{i=1}^n a_i k_i \leq \lambda}} \prod\limits_{i=1}^n k_i = \frac{\lambda^n}{n^n \prod\limits_{i=1}^n a_i} \, .
		\end{equation}
		\end{lemma}
		
		\begin{proof}
			We start by putting
			\begin{align*}
			\sup\limits_{\substack{k_1, \ldots, k_n\in {\mathbb{R}^+}\\ \sum\limits_{i=1}^n a_i k_i \leq \lambda}} \prod\limits_{i=1}^n k_i &= \frac{1}{\prod\limits_{i=1}^n a_i} \sup\limits_{\substack{k_1, \ldots, k_n\in {\mathbb{R}^+}\\ \sum\limits_{i=1}^n a_i k_i \leq \lambda}} \prod\limits_{i=1}^n a_ik_i\\
			&= \frac{1}{\prod\limits_{i=1}^n a_i} \sup\limits_{\substack{k_1, \ldots, k_n\in {\mathbb{R}^+}\\ \sum\limits_{i=1}^n k_i \leq \lambda}} \prod\limits_{i=1}^n k_i \, .
			\end{align*}
			
			We can use the fact that $\log$ is an increasing and concave function:
			
			\begin{align*}
			\sup\limits_{\substack{k_1, \ldots, k_n\in {\mathbb{R}^+}\\ \sum\limits_{i=1}^n a_i k_i \leq \lambda}} \prod\limits_{i=1}^n k_i &= \frac{1}{\prod\limits_{i=1}^n a_i} \exp \left( \sup\limits_{\substack{k_1, \ldots, k_n\in {\mathbb{R}^+}\\ \sum\limits_{i=1}^n k_i \leq \lambda}} \sum\limits_{i=1}^n \log{k_i} \right)\\
			&=  \frac{1}{\prod\limits_{i=1}^n a_i} \exp \left( \sum\limits_{i=1}^n \log(\lambda/n)\right)\\
			&= \frac{\lambda^n}{n^n \prod\limits_{i=1}^n a_i} \, .
			\end{align*}
			
			\end{proof}
			
			Now, take $\lambda >> 0$. We can rewrite equation (\ref{nodalsup1}) in the following way:

			\begin{align*}
			\mu(f_\lambda) &\leq \sup\limits_{\substack{k_1, \ldots, k_n\in {\mathbb{R}^+}\\ \sum\limits_{i=1}^n a_i (2 k_i + 1) \leq \lambda}} \prod\limits_{i=1}^n (k_i+1)\\
			&= \sup\limits_{\substack{k_1, \ldots, k_n\in {\mathbb{R}^+}\\  2\sum\limits_{i=1}^n a_i k_i  \leq \lambda-3\sum\limits_{i=1}^n a_i}} \prod\limits_{i=1}^n k_i \, .
			\end{align*}
			
			By Lemma \ref{lemmasup}, we have the following estimate for $\mu(f_\lambda)$:
			
			\begin{align}
			\label{rational1}
			\mu(f_\lambda) \leq \frac{\lambda^n}{2^n n^n \prod\limits_{i=1}^{n} a_i } + o(\lambda^n) \, .
			\end{align}

			Combining (\ref{rational1}) and (\ref{anisotropic1}), we obtain
			
			\begin{align*}
			\limsup\limits_{\lambda \to \infty}\frac{\mu(f_\lambda)}{N(\lambda)} = U(n) \, ,
			\end{align*}
			
			with $U(n)= \frac{n!}{n^n}$. 
			
			\bigskip
			
			Now, let us check that this upper bound is attained by a sequence of eigenfunctions. First, we see that

			\begin{align*}
			\sup\limits_{\substack{k_1, \ldots, k_n\in {\mathbb{Z}^+}\\  \sum\limits_{i=1}^n a_i(2k_i + 1) \leq \lambda}} \prod\limits_{i=1}^n (k_i+1) &\geq \sup\limits_{\substack{k_1, \ldots, k_n\in {\mathbb{Z}^+}\\  \sum\limits_{i=1}^n a_i(2k_i + 1) \leq \lambda}} \prod\limits_{i=1}^n (k_i)\\
			&= \sup\limits_{\substack{k_1, \ldots, k_n\in {\mathbb{R}^+}\\  2\sum\limits_{i=1}^n a_i k_i  \leq {\lambda-\sum\limits_{i=1}^n a_i}}} \prod\limits_{i=1}^n k_i \, .
			\end{align*}
			
			We use Lemma \ref{lemmasup} to obtain:
			
			\begin{align*}
			\sup\limits_{\substack{k_1, \ldots, k_n\in {\mathbb{Z}^+}\\  \sum\limits_{i=1}^n a_i(2k_i + 1) \leq \lambda}} \prod\limits_{i=1}^n (k_i+1) \geq \frac{\lambda^n}{2^n n^n \prod\limits_{i=1}^{n} a_i } + o(\lambda^n) \, .
			\end{align*}
			
			This means that for every $\lambda > 0$, there exists an eigenfunction $f_\lambda$ such that \linebreak $$\mu(f_\lambda) \geq \frac{\lambda^n}{2^n n^n \prod\limits_{i=1}^{n} a_i } + o(\lambda^n) \, .$$
			
			We can then construct a sequence of eigenfunctions $f_{n_k}$ such that $$\limsup\limits_{k\to \infty} \frac{\mu(f_{n_k})}{n_k} = U(n) \, .$$ This shows that $U(n)$ is indeed optimal, which completes the proof of theorem \ref{theoremeprincipal2}.
			
			\bigskip

			 Let us compare $U(n)$ with $\gamma(n)$:
			
			\begin{proposition}
				For all $n \geq 2$, $U(n) < \gamma(n)$. Furthermore, $$ \frac{\gamma(n)}{U(n)} >  2^{n-\frac{5}{2}} \sqrt{\pi n} e^{-2\sqrt{n}} (1 + o_n(1)) \, $$ as $n$ goes to infinity.
				
				\bigskip
				
				Therefore, $U(n)$ decays much faster than $\gamma(n)$ as $n$ goes to infinity.
			\end{proposition}
			
			\begin{proof}
				We start by putting
\begin{align*}
\frac{\gamma(n)}{U(n)} = \frac{2^{n-2}n^2 \Gamma(\frac{n}{2})^2 n^n}{n! (j_{\frac{n}{2}-1})^n} \, .
\end{align*}			
			
			If $n=2k$, we have
			
		\begin{align*}
		{\frac{\gamma(n)}{U(n)}} &= 	\frac{2^{2k-2}(2k)^2 \Gamma(k)^2 (2k)^{2k}}{(2k)! (j_{k-1})^{2k}}\\
		&= \frac{2^{4k} (k!)^2 k^{2k}}{(2k)! (j_{k-1})^{2k}} \, .
		\end{align*}
			
		It is shown in \cite{HelfferPersson} that for $u > 0$, $\sqrt{u(u+2)} < j_u < \sqrt{u+1}(\sqrt{u+2}+1)$. Hence, for $u>10$, $j_{u-1} < \sqrt{2}u$. Also, $(2k)! <  2^{3k} (k!)^2$ for $k \geq 1$. Combining those two facts with the previous equation, we get for $k > 10$
			
			\begin{align*}
			{\frac{\gamma(n)}{U(n)}} = \frac{2^{3k}(k!)^2}{(2k)!} \frac{(\sqrt{2}k)^{2k}}{(j_{k-1})^{2k}} > 1	\, .		
			\end{align*}
			
		If $n=2k+1$, we have
		
		\begin{align*}
		{\frac{\gamma(n)}{U(n)}} = 	\frac{2^{2k-1}(2k+1)^2 \Gamma(k+1/2)^2 (2k+1)^{2k+1}}{(2k+1)! (j_{k-1/2})^{2k+1}}
		\end{align*}
		
		Using the identity $\Gamma(k+ 1/2) = \frac{(2k)!}{4^k k!}\sqrt{\pi}$, we get

		\begin{align*}
		{\frac{\gamma(n)}{U(n)}} &= \frac{\pi 2^{2k-1}(2k+1)^2 ((2k)!)^2 (2k+1)^{2k+1}}{ 2^{4k}(k!)^2 (2k+1)! (j_{k-1/2})^{2k+1}}\\
		&=	\frac{\pi (2k+1)! (2k+1)^{2k+1} }{(j_{k-1/2})^{2k+1} 2^{2k+1}(k!)^2} \, .
		\end{align*}
		
		We use the fact that $(2k+1)! > 2^{2k}(k!)^2$ and that $j_{u-1/2} < \sqrt{2}(u-1/2)$ for $u > 10$ to obtain for $k > 10$ that
		
		\begin{align*}
		{\frac{\gamma(n)}{U(n)}} = \frac{(2k+1)!}{2^{2k}(k!)^2} \frac{(2k+1)^{2k+1} }{(j_{k-1/2})^{2k+1}} \frac{\pi}{2} > 1 \, .
		\end{align*}
		
		We only need to check that $\gamma(n) > U(n)$ for $n=1,2,\ldots, 21$, which is done using Mathematica.
		
		\bigskip
		
		Now, using Stirling's formula and the estimate $j_{\frac{n}{2}-1} \leq \left( \sqrt{\frac{n}{2}} + \sqrt{\frac{1}{2}}  \right)^2$, we have the following:
		
		\begin{align*}
			{\frac{\gamma(n)}{U(n)}} &> \frac{2^{n-2}\Gamma(\frac{n}{2}+1)^2 n^n}{\left( \sqrt{\frac{n}{2}} + \sqrt{\frac{1}{2}}  \right)^2 n!}\\
			&= \frac{2^{n-2} (\frac{n}{2e})^n \pi n (1 + o_n(1))  \sqrt{n}^{2n} }{\left( \sqrt{\frac{n}{2}} + \sqrt{\frac{1}{2}}  \right)^{2n} (\frac{n}{e})^n \sqrt{2\pi n}(1 + o_n(1)) }\\
			&= \frac{2^{n-\frac{5}{2}} \sqrt{\pi n}}{ \left( 1 + \sqrt{\frac{1}{n}}  \right)^{2n}}(1 + o_n(1))
		\end{align*}
		
	Now, we use the fact that $\left( 1 +\sqrt{\frac{1}{n}}  \right)^{\sqrt{n}} < e$ to obtain as $n$ goes to infinity
	
	\begin{align*}
	{\frac{\gamma(n)}{U(n)}} &> 2^{n-\frac{5}{2}} \sqrt{\pi n} e^{-2\sqrt{n}} (1 + o_n(1)) \, .
	\end{align*}
	
		\end{proof}
		
\begin{remark}
	It is clear that the constant $\gamma(n)$ can be improved for the quantum harmonic oscillator. It is still unknown if the constant $U(n)$ is the optimal constant in the general case. There is a similar question concerning Pleijel's theorem for the Dirichlet or Neumann Laplacian. In the case of an irrationnal rectangle, the constant $\gamma(n)$ can be lowered to $\frac{2}{\pi}$. It has been conjectured by I. Polterovich in \cite{Polterovich} that $\frac{2}{\pi}$ is the optimal constant for any planar domain.
\end{remark}

			\bigskip

\subsection*{Thanks}This paper is based on a MSc thesis written under the supervision of Iosif Polterovich. I would like to thank I. Polterovich for his guidance, as well as Bernard Helffer for suggesting the original problem and numerous helpful remarks. This article was completed while I was visiting the Laboratoire Jean Leray in Nantes, and its hospitality is gratefully acknowledged. I would also like to thank Guillaume Roy-Fortin for his comments and suggestions. This research was partially supported by the FRQNT.

\centering{Philippe Charron\\
	Université de Montréal\\
	2920, Chemin de la Tour\\
	Montréal, QC \\
	H3T 1J4}

\end{document}